\documentclass[11pt]{amsart}

\usepackage{amsmath}
\usepackage{amsthm}
\usepackage{amssymb}
\usepackage{eucal}
\usepackage{mathrsfs}
\usepackage{graphicx}

\theoremstyle{definition}
\newtheorem{thm}{Theorem}[section]   
\newtheorem{lem}[thm]{Lemma}         

\newtheorem{defi}[thm]{Definition}   

\renewcommand{\phi}{\varphi}

\newcommand{\N}{{\mathbb N}}

\renewcommand\phi{\varphi}
\newcommand\by{\mbox{$\times$}}
\numberwithin{equation}{section}

\DeclareMathOperator{\tr}{tr}

\begin{document}


\title{Joint spectral radius and forbidden products}
\author{Alexander Vladimirov}
\thanks{Institute for Information Transmission Problems, Moscow}

\address{Institute for Information Transmission Problems, Moscow}
\email{vladim@iitp.ru}


\maketitle

\begin{abstract}
We address the problem of finite products that attain the joint spectral radius of a finite number of square matrices. Up to date the problem of existence of ``forbidden products'' remained open. We prove that the product $AABABABB$ (together with its circular shifts and their mirror images) never delivers the strict maximum to the joint spectral radius if we restrict consideration to pairs $\{A,B\}$ of real $2\by 2$ matrices. Under this restriction circular shifts and their mirror images constitute the class of isospectral products and hence they all have the same spectral radius for any pair $\{A,B\}$ of $2\by 2$ matrices, even complex. For pairs of complex matrices we have numerical evidence  that $AABABABB$ is still a fobidden product. A couple of binary words that encode products from this isospectral class also happen to be the shortest forbidden patterns in the parametric family of double rotations.
\end{abstract}

\section{Joint spectral radius}

Let $S$ be a set of $d\by d$ matrices, real or complex. By $S^n$ we denote the collection of products $A_1\dots A_n$ where all $A_i\in S$ (with possible repetitions). 

Denote by $\rho(A)$ the spectral radius of $A$, that is, the maximal absolute value of eigenvalues of $A$.  We define the {\em joint spectral radius} (or JSR) of $S$ as 
$$
\rho(S) = \sup_{n\in\N,\Pi\in S^n}\rho(\Pi)^{\frac{1}{n}},
$$
This is not a standard definition of JSR, though it is equivalent to the standard one, see \cite{RotA6037,DauSe9222,LagTh9517,BloTh9731,BloTh2008,Jun09} for the history of the subject. 

We will say that $\rho(A_1\dots A_n)^{\frac{1}{n}}$ is the {\em normalized spectral radius} of the product $\Pi=A_1\dots A_n$. Hence, $\rho(S)$ is the supremum of normalized spectral radii of finite products of matrices from $S$. 

We say that the {\em finiteness property} holds for $S$ if, for some $n$, there exists a finite product $\Pi=A_1\dots A_n\in S^n$ such that $\rho(\Pi)^{\frac{1}{n}}=\rho(S)$. In this case the product $\Pi$ is called a {\em spectrum maximizing product} or SMP for the collection $S$ of square matrices. 

As an example, take
$$
A=\left(\begin{array}{cc} 0 & 1\\0 & 0\end{array}\right),\quad B=\left(\begin{array}{cc} 0 & 0\\1 & 0\end{array}\right).
$$
We get 
$$
AA=BB=\left(\begin{array}{cc} 0 & 0\\0 & 0\end{array}\right),\quad AB=\left(\begin{array}{cc} 1& 0\\0 & 0\end{array}\right),\quad BA=\left(\begin{array}{cc} 0 & 0\\0 & 1\end{array}\right).
$$

Clearly,
$$
\rho(A)=\rho(B)=0,\quad \rho(AB)=1,\quad \rho(\{A,B\})=1.
$$
The products $AB$, $BA$ and their powers are SMPs for the set $S=\{A,B\}$.

\section{Binary words}

In this paper we mainly consider pairs of $2\by 2$ matrices, that is, $S=\{A,B\}$, $d=2$. Let us define a homomorphism $P(W)$ of free semigroup $W$ on the alphabet $\{1,2\}$ to the semigroup of finite products of matrices $A$ and $B$ (the empty product is not included). We prefer to use symbols $1$ and $2$ instead of conventional $0$ and $1$ or $a$ and $b$ in order to be able to interpret the words as positive integers. 

We set $P(1)=A$, $P(2)=B$ and $P(w_1 w_2)=P(w_1)P(w_2)$ for any pair $w_1,w_2\in W$. That is, we substitute $A$ for $1$ and $B$ for $2$ in the binary word $w$.

Accordingly, we will use the term SMP for binary words as well as soon as a pair of matrices $\{A,B\}$ is given. Also we will call two words $w_1,w_2\in W$ {\em isospectral} if the products $P(w_1)$ and $P(w_2)$ are isospectral (have the same spectrum) for any pair of $2\by 2$-matrices. It is clear that isospectral words must have the same length. Moreover, they must be permutations of each other.

As is known, all cyclic shifts of $\Pi$ have the same spectrum and hence the same spectral radius as $\Pi$.
Hence if $\Pi$ is an SMP for $S$ then all its cyclic shifts are SMPs for $S$. This property holds for all dimensions $d$ but, for the particular case $d=2$, there exist other isospectral products. For instance, the mirror
image of $w$ is isospectral to $w$ \cite{southcott1979trace}, see also \cite{bochi2023spectrum} for comprehensive review of isospectrality results and for further references. The mirror image is obtained by writing the digits of $w$ in opposite order.

In the study of JSR, one of the main open problems is if the finiteness property has probability one among finite sets of matrices. As is a common case in ergodic optimization, the optimal paths, that is, SMPs, tend to be rather short periodic ones, see \cite{hunt1996optimal,contreras2016ground}. In the case of matrix products, it is widely believed that short periodic products deliver the maximum of growth rate, that is, they are SMPs in a majority of cases. 

Numerical testing confirms this guess. Experiments with random pairs of $2\by 2$-matrices whose elements are independently normally distributed produce frequency lists for finite binary words, and, in general, the frequency of a word being an SMP drops sharply with the length of the word. 

The shortest pair of isospectral words that cannot be produced from each other by operations of cyclic shifts and mirror image has length $12$. The general rule for isospectral binary words is currently not available \cite{bochi2023spectrum}. 

For square matrices of dimension $3$ and higher, it is not known if there exist isospectral words different from cyclic shifts of each other. If they exist, they should be longer than $30$, see \cite{bochi2023spectrum}.

We represent the set $W$ of primitive binary words as a disjoint union of isospectral clusters. Each cluster is a finite set of words whose normalized spectral radii coincide for each pair of $2\by 2$ matrices. Recall that a primitive word is a word that is not a power of a shorter word. Say, $12$ is primitive and $121212$ is not.

We will say that a primitive word $w$ is a {\em unique SMP} for a given pair of matrices $\{A,B\}$ if its normalized spectral radius is strictly larger than that of any primitive word $u$ that is not isospectral to $w$. Our main result here is the existence of a word that is not a unique SMP for any pair of real $2\by 2$ matrices.

\section{List of words}

In what follows we will only study binary words of length $L\le 8$ and respective products of $2\by 2$ matrices, real and complex. There are $62$ isospectral clusters of primitive words of this kind. For each cluster we choose a {\em representative word}. This is a word that is minimal in numerical order among all words of this cluster. For instance, the cluster $112, 121, 211$ is represented by the Lyndon word $112$.

We write the representative words of all isospectral clusters in the numerical order. The representative words of length $L\le 8$ constitute the list
$$
\begin{array}{lllll}
 1  &   2 &   12 &  112 &  122 \\
        1112    &    1122  &      1222  &     11112 &      11122\\
       11212   &    11222    &   12122   &    12222   &   111112\\
      111122    &  111212   &   111222    &  112122   &   112222\\
      121222    &  122222  &   1111112   &  1111122   &  1111212\\
     1111222   &  1112112   &  1112122   &  1112222  &   1121122\\
     1121212   &  1121222   &  1122122   &  1122222   &  1212122\\
     1212222   &  1221222   &  1222222  &  11111112  &  11111122\\
    11111212   & 11111222  &  11112112  &  11112122   & 11112222\\
    11121122   & 11121212   & w_{48}=11121222  &  11122122   & 11122222\\
    11211212   & 11211222 &   w_{53}=11212122  &  11212212  &  11212222\\
    11221222  &  11222222 &   12121222 &   12122122  &  12122222\\
    12212222   & 12222222&&&
\end{array}
$$
with $62$ entries. The words $w_{48}$ and $w_{53}$ are marked since they apparently never happen to deliver the unique maximum of normalized spectral radius among the words in the list. For $w_{53}$ we will prove this fact for pairs of real matrices. 

Let us give a brief description of the computer test. We generate $10^9$ random pairs $\{A,B\}$ of $2\by 2$ matrices. Each element of $A$ and $B$ is taken normally distributed and they are mutually independent. Then we calculate the normalized spectral radii of $62$ products and compare them.

It happens in my computer testings that, for each pair $\{A,B\}$, the maximal value of normalized spectral radius is attained at a single word among these $62$ representative words.  In the frequency list with $62$ entries, we add a unit to $m$-th entry each time the word $w_m$ provides the maximum of normalized spectral radius. 

If a word $w$ and a pair of matrices $\{A,B\}$ contribute to the frequency list, that does not imply, of course, that this word is SMP for the pair $\{A,B\}$ since we do not take maximum among all finite words, just among $62$ words in the list. The reverse implication is, however, true, that is, if a word of length $8$ or less is an SMP for $\{A,B\}$, then it delivers a maximum within the list of $62$ words.

\section{Frequency lists}

Let us look at two frequency lists for random real matrices and for random complex matrices. They are
$$
\begin{array}{lllll}
   380926385 &  380935818  & 130639753  &  22203148 &   22222524\\
    8830761 &  2566104 &    8830465 &   4925761 &   1125430\\
    2252074  &   1125690 &    2248037  &   4923084  &   2656504\\
    493994  &   94050  &  406132 & 11930 &   493757\\
    94025 &    2655848  &   1970137 &  312701 &  202762\\
    173712  &   1003228  &  3726 &   173514 &  423694\\
     1024319  &  3748 &   423266 &  312412  &  1021596\\
     202272 & 1003841  &   1974220 &  2804615 &  433927\\
     27866  &  252274 &  167688  &  682 &  173212\\
     10424  &   54485 &   0 &  14716 &  252269\\
     639544  &     14937  &         0  &    132280 &         716\\
      10419   &   434665  &     53608  &    637135   &    27811\\
      167459   &  2798846 &&&
      \end{array}
$$
and
$$
\begin{array}{lllll}
   409778851  & 409788750 &  115014092 &   20572425  &  20577206\\
    3996733 &     2812670 &    3999407 &    1367670  &    595576\\
      2101481 &     596830 &     2101519 &    1367205  &    472635\\
168696   &    24346  &    113651 &         427    &  168296\\
  24620 &     472898  &    244202   &    68458 &       52915\\
   30537 &     368549  &        85  &     30989   &   168740\\
      409995   &       93  &    168746  &     68580   &   409738\\
       52148  &    367624 &     243776   &   207092    &   53031\\        
       1458   &    21026 &       41219  &         2    &   13237\\
        239  &      7436  &         0 &        1893  &     20739\\
             218661  &      1860  &         0 &      81939  &           2\\
                      254  &     52751   &     7603   &   218570     &   1412\\
       40849  &    207568 &&&
\end{array}
$$

Surely, the words $1$ and $2$ are the most frequent ones, and then goes the word $12$, the words $112$ and $122$, and so on, though the frequency does not decrease monotonically as the length of the word increases. We guess that some kind of complexity parameter of finite words might be responsible for their frequency. 

The percentage of optimal words with lengths ranging from $1$ to $8$ (for real matrices) looks like this:
$$
69.8659\quad  23.9591\quad  4.0746\quad  1.2363\quad  0.5073\quad  0.1582\quad  0.1172\quad  0.0813
$$

\section{Forbidden words}

Notably, two words have zero frequencies in both lists. They are $w_{48}=11121222$ and $w_{53}=11212122$ that are marked in the list of representative words. According to extended numerical tests, for lengths over $8$, there seem to be many words of this kind.

I have run analogous tests for $3\by 3$ matrices for words of length up to $16$ and found that every word has positive frequency.

I also tried the {\em joint spectral subradius} instead of JSR and found no candidates to forbidden words even for $2\by 2$ matrices.

I also tested $2\by 2$ matrices with objective function $\|\Pi\|^{1/n}$ with different matrix norms instead of $\rho(\Pi)^{1/n}$ and, again, found no candidates to forbidden words.

Here we will prove that $w_{53}=112121222$ never delivers a single maximum (up to cyclic shifts and mirror images) to the normalized spectral radius among all the words of length $8$ or less. This implies, of course, that $w_{53}$ is not a unique SMP for any pair of $2\by 2$ matrices $\{A,B\}$. My numerical study suggests that this property holds for pairs of complex matrices as well, though we are only able to prove it for real ones. 

\begin{thm}\label{T1}
For any pair of real $2\by 2$ matrices $\{A,B\}$, the product 
$$
w_{53}(\{A,B\})=AABABABB
$$ 
never delivers a single maximum (up to cyclic shifts and mirror images) to the normalized spectral radius among all the products of length $8$ or less.
\end{thm}

\section{Spectral radius as a function of trace}

For a while, we consider only matrices with determinant one

\begin{lem}\label{L1}
If $A$ and $B$ are real $2\by 2$-matrices and if $\det(A)=\det(B)=1$, then $\rho(B)>\rho(A)$ if and only if
$$
|\tr(B)|>\max\left\{2,|\tr(A)|\right\}.
$$
\end{lem}
\begin{proof}
Let $A$ be a real $2\by 2$-matrix and $\det(A)=1$. Let $\lambda$ and $1/\lambda$ be the eigenvalues of $A$. Then $\lambda+1/\lambda=t$, that is,
\begin{equation}\label{lsq}
\lambda^2-\lambda t+1=0.
\end{equation}
If $-2<t<2$ then the roots of (\ref{lsq}) are complex conjugate numbers of equal absolute value. This value must be equal to $1$ since $\det(A)=1$. 

For $t\ge 2$, $\rho(A)$ is equal to the larger root of (\ref{lsq}) and therefore it is strictly increasing from $1$ to $\infty$ as $t$ varies from $2$ to $\infty$. It remains to notice that $\tr(B)=-\tr(A)$ implies $B=-A$. Hence Lemma \ref{L1} holds.
\end{proof}

\section{Reduced problems}

Instead of comparing the product $w_{53}(A,B)$ with all other $61$ products in the list of representative words, I found numerically just $3$ words from this list that apparently always dominate the word $w_{53}$. They are $w_3=12$, $w_7=1122$, and $w_{54}=11212212$. By domination we mean that, for any pair of $2\by 2$ matrices, the maximum of normalized spectral radii of these three words is greater or equal to the normalized spectral radius of $w_{53}$.

Given a pair of $2\by 2$-matrices, let us restrict consideration to the following products of length $8$: $P_1=ABABABAB$, $P_2=AABBAABB$, $P_3=AABABBAB$, and $P_4=AABABABB$. Our goal is to prove that 
\begin{equation}\label{mne}
\rho(P_4)\le \max\{\rho(P_i): i=1,2,3\},
\end{equation}
for all pairs of real matrices $\{A,B\}$. This would imply Theorem \ref{T1}.

Note that there exist circular shifts of all four words that can be represented as products of two matrices $C=AB$ and $D=BA$. They are, for instance, $ABABABAB$, $ABBAABBA$, $ABABBABA$, and $ABABABBA$. We now study the spectral radii of four products $Q_1=CCCC$, $Q_2=CDCD$, $Q_3=CCDD$, and $Q_4=CCCD$.

\section{Fricke polynomials}

We will need {\em Fricke polynomials} for the proof of our domination result.

For introduction to Fricke polynomials, see, for instance, \cite{bochi2023spectrum} where they are called reduced Fricke polynomials. These are polynomials of three variables $x,y,z$ (traces of $A$, $B$, and $AB$, respectively, denoted $\tr(A),\tr(B),\tr(AB)$) with integer coefficients. For each finite binary word $w$ there exists a unique Fricke polynomial $F_w(x,y,z)$ such that $\tr(P_w(A,B))=F_w(\tr(A),\tr(B),\tr(AB))$ for each pair $\{A,B\}$ of complex $2\by 2$-matrices satisfying the condition $\det(A)=\det(B)=1$, see \cite{bochi2023spectrum} for detailed review and further references.

For matrices with other determinants, the trace of any finite product can be easily found by scaling. 

Let $C$ and $D$ be complex matrices with $\tr(C)=\tr(D)$ and $\det(C)=\det(D)=1$. Then we turn to Fricke polynomials in variables $x,y,z$ (they can be obtained from an online resource):
\begin{eqnarray*}
FR_1&=&2-4x^2+x^4,\\
FR_2&=&-2+z^2,\\
FR_3&=&2-x^2-y^2+xyz,\\
FR_4&=&-xy-z+x^2z.
\end{eqnarray*}

We have a special case $x=y$ since the traces of $AB$ and $BA$ are equal. Moreover, for all four products $P_i$, the Fricke polynomials can be represented as polynomials of two variables $u=x^2$ and $z$, where $x$ is the trace of $C$ and $z$ is the trace of $CD$.

Let us now write down the resulting four polynomials in variables $u,z$:
\begin{eqnarray}
\label{FF1} FS_1(u,z)&=&2-4u+u^2;\\ 
FS_2(u,z)&=&z^2-2;\\ \label{FF2}
FS_3(u,z)&=&2-2u+uz;\\ \label{FF3}
FS_4(u,z)&=&uz-u-z; \label{FF4}
\end{eqnarray}

Our next goal is to prove inequality \eqref{mne} under an additional constraint on matrices $C$ and $D$. Actually, we will prove the inequality
\begin{equation}\label{mne1}
\rho(Q_4)\le \max\{\rho(Q_i): i=1,2,3\}, 
\end{equation}
for products $Q_1=CCCC$, $Q_2=CDCD$, $Q_3=CCDD$, and $Q_4=CCCD$.

\section{Inequalities}

\begin{lem}\label{L2}
Let $C$ and $D$ be complex $2\by 2$ matrices with $\tr(C)=\tr(D)$ and $\det(C)=\det(D)=1$, and such that the values $u=\tr(C)^2$ and $z=\tr(CD)$ are real. Then inequality (\ref{mne1}) holds.
\end{lem}

In the system (\ref{FF1}--\ref{FF4}), let us make a change of variable $u=v+2$. We get
\begin{eqnarray*}
F_1(v,z)&=&v^2-2;\\
F_2(v,z)&=&z^2-2;\\
F_3(v,z)&=&vz+2z-2v-2;\\
F_4(v,z)&=&vz+z-v-2.
\end{eqnarray*}

Now, let us look at the differences $G_i(v,z)=F_i(v,z)-F_4(v,z)$ for $i=1,2,3$. We get
\begin{eqnarray}
G_1(v,z)&=&(-v-1)(z-v); \label{G11}\\
G_2(v,z)&=&(z-1)(z-v); \label{G12}\\
G_3(v,z)&=&z-v. \label{G13}
\end{eqnarray}

If (\ref{mne1}) breaks down, that is, if $|F_4(v,z)|> \max\{|F_i(v,z)|: i=1,2,3\}$, then either $G_i(v,z)>0$ for $i=1,2,3$, or $G_i(v,z)<0$ for $i=1,2,3$. In the second case we have $z<v$ from (\ref{G13}), $v<-1$ from (\ref{G11}), and $z>1$ from (\ref{G12}) which is a contradiction. 

It remains to consider the first case, that is, $v<-1$, $z>1$. In this case, if (\ref{mne1}) breaks down, then we must have $F_4<0$ and $F_i<-F_4$. On the domain $\{v,z:v<-1,z>1\}$, let us find areas where each polynomial $F_i(v,z)$, $i=1,2,3$, is maximal among $F_1(v,z),F_2(v,z),F_3(v,z)$. 

\begin{figure}[h]\label{fw3}
\centerline{\includegraphics[height=10cm]{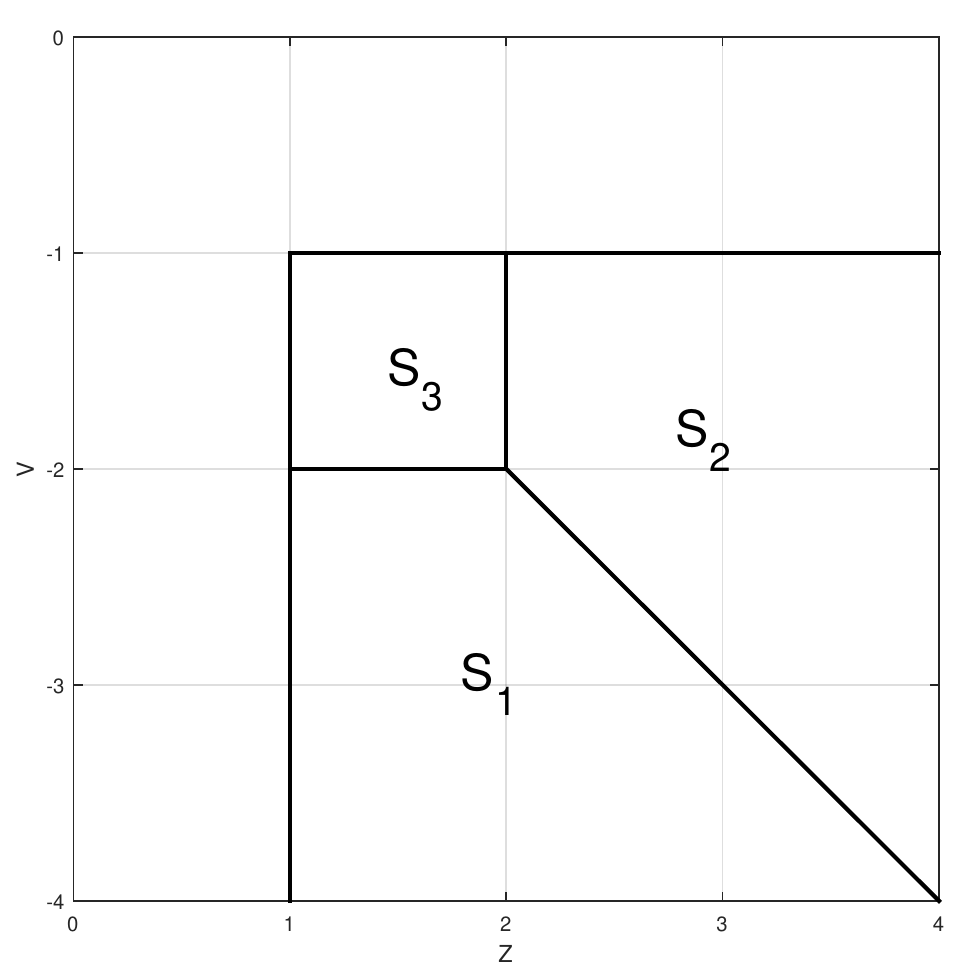}}
\caption{Areas of maximality}
\end{figure}

The inequality 
$$
F_3\ge F_1\quad \iff\quad (v+2)(v-z)\le 0
$$
is equivalent to $v\le -2$ since $v-z$ is negative on the domain $\{v,z:v<-1,z>1\}$. The inequality
$$
F_3\ge F_2\quad \iff\quad (z-2)(z-v)\le 0
$$
is equivalent to $z\le 2$. The intersection with the domain $\{v,z:v<-1,z>1\}$ gives us the square 
$$
S_3=\{v,z:-2\le v\le -1,1\le z\le 2\},
$$
where $F_3(v,z)$ is maximal.

Since the sum $F_3+F_4$ is linear both in variable $v$ and in variable $z$, it suffices to check that $F_3+F_4\ge 0$ at all the vertices of the square. This is, indeed, the case.

Then we look at two other areas
$$
S_1=\{(v,z):v\le -2,\, 1\le z\le -v\},
$$
and
$$
S_2=\{(v,z):z\ge 2,\, -z\le v\le -1\}
$$ 
where $F_1(v,z)$ and $F_2(v,z)$ are maximal, respectively.

They are convex polygonal sets, both unbounded. We take advantage of linearity of the sum $F_1+F_4$ in variable $z$ and linearity of the sum $F_2+F_4$ in variable $v$. We see that $F_1+F_4$ is not negative on the rays $(-2,1)+\alpha(-1,0)$ and $(-2,2)+\alpha(-1,-1)$, $\alpha\ge 0$. This suffices to prove that $F_1+F_4$ is not negative on the whole set $S_1$. Analogously, we prove that $F_2+F_4$ is not negative on the whole set $S_2$. Thus Lemma \ref{L2} is proved.

Note that the second condition of Lemma \ref{L2} holds if matrices $C$ and $D$ are real as well as if they are purely imaginary.

\section{General values of determinants}

Let us return to the products $P_1=ABABABAB$, $P_2=AABBAABB$, $P_3=AABABBAB$, and $P_4=AABABABB$. If we consider matrices $\tilde A=aA$ and $\tilde B=bB$ instead of $A$ and $B$, respectively, for arbitrary non-zero complex $a,b$, this substitution does not change the order of spectral radii of our four products since they all are multiplied by the same number $|a^4 b^4|$. 

Suppose now that $\det(AB)\ne 0$. Then there exist $a$ and $b$ such that $ab$ is either real or purely imaginary number and $\det(\tilde{A}\tilde{B})=1$. Clearly conditions of Lemma  \ref{L2} hold for matrices $C=\tilde{A}\tilde{B}$ and $D=\tilde{B}\tilde{A}$ and we conclude that inequality (\ref{mne}) holds.

In the case $\det(AB)=0$, we pass to the limit as $A_i\to A$ and $B_i\to B$ and $\det(A_i B_i)\ne 0$. This concludes the proof of Theorem \ref{T1}.

We have only proved that $w_{53}$ is forbidden in the case of real matrices, though numerical tests suggest that this is also true for complex matrices. At least, we have not found any counterexample among several billions of random pairs of complex $2\by 2$-matrices, each one of their $8$ elements drawn according to the normal distribution, independently.

\section{Forbidden patterns}

\begin{defi}
We say that a binary word $w$ is a {\em forbidden pattern} if each primitive word $uwv$ is forbidden for $2\by 2$ matrices, where $u$ and $v$ are finite binary words, may be, empty.
\end{defi}

We do not know if forbidden patterns exist. Numerical results suggest (but do not prove) that neither $w_{53}$ nor $w_{48}$, nor their circular shifts, nor circular shifts of their mirror images are forbidden patterns. The existence of forbidden patterns would imply that the fraction of forbidden words among all binary words of length $N$ tends to  $1$ as $N$ tends to $\infty$. We conjecture that this is indeed the case.

\section{Double rotations}

Here we consider a simple parametric dynamical system, {\em double rotation}, see \cite{suzuki2005double,clack2013double,zhuravlev2010one,kryzhevich2021dynamics,gorodetski2016synchronization}. If we encode the paths of this system by binary words, the list of words that are never produced by these paths begins with a circular shift of the mirror image of the word $w_{53}$, namely, with the word $12121122$. Note that this proposition is in some sense stronger than Theorem 1: An analogous result for JSR would mean that $12121122$ is not just a forbidden word, but a forbidden pattern, but this seems to be wrong.

Double rotations also had been referred by Victor Kozyakin \cite{kozyakin2022non} in connection with Barabanov norms of pairs of $2\by 2$ matrices.

Let us consider the following dynamical system with discrete time (double rotation). Real positive $R<1$ and real $h_1,h_2$ are given. The state space is the interval $[0,1)$. The map is 
$$
f(x)= \{x+h_1\}
$$
if $x<R$ and
$$
 f(x)=\{x+h_2\}
$$
otherwise. Here $\{x\}$ is the fractional part of $x$.

With each path $x_1,x_2,\dots,x_n$ we associate a binary word: we put $1$ if $x_i<R$ and $2$ otherwise. The question is which finite binary words cannot be produced by any set of parameters $R,h_1,h_2$. Numerical testing demonstrates that all the words of length $L\le 8$ can be produced in this way apart from the words $12121122$ and $21212211$ (yes, they are isospectral to $w_{53}$).

Let us prove that the word $w=12121122$ cannot be produced by any double rotation. Suppose there exists a path $x_1,\dots,x_8$ on the interval $[0,1)$ with associated binary word $w$, that is, 
$$
x_1,x_3,x_5,x_6\in[0,R),\quad x_2,x_4,x_7,x_8\in[R,1).
$$
Denote $\alpha=x_4-x_2$. Suppose first that $\alpha>0$. Clearly, 
\begin{equation}\label{E77}
x_3-x_1=x_5-x_3=x_8-x_6=\alpha.
\end{equation}
We also have $x_6=x_4+\alpha-1$, hence $0\le x_6<\alpha\le x_3$. Then, from (\ref{E77}), we get $x_8<x_5<R$ which is a contradiction since $x_8\ge R$. The case $\alpha<0$ is handled in a similar way.    

We have proved that  $12121122$ is a forbidden pattern. Hence its complementary word $21212211$ is also a forbidden pattern. All the other binary words of length $8$ or less were generated numerically for some combinations of parameters $R,h_1,h_2$.

It follows immediately that all the $8$-periodic sequences generated by the words $12121122$ and $21212211$ do not correspond to periodic solutions of any double rotation. By numerical testing we certify that all other binary periodic sequences of period $8$ or less are possible.

\newcommand{\etalchar}[1]{$^{#1}$}

 
\end{document}